\documentclass[12pt]{amsart}
\usepackage{amscd, amssymb}
\input{xy}
\xyoption{all}

\newtheorem{Thm}{Theorem}[subsection]
\newtheorem{Conj}[Thm]{Conjecture}
\newtheorem{Prop}[Thm]{Proposition}
\newtheorem{Def}[Thm]{Definition}
\newtheorem{Def/Thm}[Thm]{Definition/Theorem}
\newtheorem{Cor}[Thm]{Corollary}
\newtheorem{Lemma}[Thm]{Lemma}

\theoremstyle{remark}
\newtheorem{Rmk}[Thm]{Remark}

\newtheorem{EG}[Thm]{Example}

\numberwithin{equation}{subsection}

\newcommand{\ot }{\otimes}
\newcommand{\ra }{\rightarrow}

\newcommand{\lra }{\longrightarrow}

\newcommand{\Hom }{{\mathrm{Hom}}}

\newcommand{\Spec}{{\mathrm{Spec}}}

\newcommand{\Pic}{{\mathrm{Pic}}}

\newcommand{\cO}{{\mathcal{O}}}

\newcommand{\cL}{{\mathcal{L}}}
\newcommand{\cE}{{\mathcal{E}}}
\newcommand{\cF}{{\mathcal{F}}}

\newcommand{\cP}{{\mathcal{P}}}
\newcommand{\cQ}{{\mathcal{Q}}}
\newcommand{\cV}{{\mathcal{V}}}
\newcommand{\cC}{{\mathcal{C}}}

\newcommand{\cT}{{\mathcal{T}}}
\newcommand{\cX}{{\mathcal{X}}}

\newcommand{\bT}{{\bf T}}

\newcommand{\PP }{{\mathbb P}}
\newcommand{\GG }{{\mathbb G}}
\newcommand{\QQ }{{\mathbb Q}}
\newcommand{\CC }{{\mathbb C}}
\newcommand{\ZZ }{{\mathbb Z}}
\newcommand{\RR }{{\mathbb R}}

\newcommand{\Mgk}{\overline{M}_{g,k}}
\newcommand{\fMgk}{\mathfrak{M}_{g,k}}

\newcommand{\cTX}{\mathcal{T}_{g,k}(X_\Sigma,\beta)}
\newcommand{\cTXP}{\cT_{0,k}(X_\Sigma,\beta; \PP^1)}

\newcommand{\fS}{\mathfrak{S}}

\begin{document}
\title{Moduli stacks of stable toric quasimaps}

\begin{abstract} We construct new ``virtually smooth" modular compactifications of spaces of
maps from nonsingular curves to smooth projective toric varieties. They generalize Givental's
compactifications, when the complex 
structure of the curve is allowed to vary and markings are included, and are
the toric counterpart of the moduli spaces of stable quotients introduced by Marian, Oprea, and Pandharipande
to compactify spaces of maps to Grassmannians. A brief discussion of the resulting invariants and their
(conjectural) relation with Gromov-Witten theory is also included.
\end{abstract}
\subjclass[2000]{14D20, 14D23, 14M25, 14N35}

\author{Ionu\c t Ciocan-Fontanine}
\noindent\address{School of Mathematics, University of Minnesota, 206 Church St. SE,
Minneapolis MN, 55455, and\hfill
\newline \indent School of Mathematics, Korea Institute for Advanced Study,
87 Hoegiro, Dongdaemun-gu, Seoul, 130-722, Korea}
\email{ciocan@math.umn.edu}

\author{Bumsig Kim}
\address{School of Mathematics, Korea Institute for Advanced Study,
87 Hoegiro, Dongdaemun-gu, Seoul, 130-722, Korea}
\email{bumsig@kias.re.kr}

\maketitle

\section{Introduction} The Gromov-Witten theory of smooth toric varieties, 
and that of complete intersections of ample divisors in them, has been extensively studied. 

The genus zero theory over the small parameter space has a completely determined answer, expressed
in terms of Givental's {\it $I$-functions}. This
is shown in many cases in the original work \cite{G}; in full generality, it follows from further results
in \cite{CG}, \cite{Iritani}, see also \cite{Brown}.

The $I$-function of a smooth projective toric
variety $X_\Sigma$ (associated to complete nonsingular fan $\Sigma$) was introduced by Givental in \cite{G} in
terms of ``toric compactifications" of the moduli spaces of maps from $\PP^1$ to $X_\Sigma$.
These compactifications, also introduced and studied in \cite{MP} for the same purpose of understanding
Gromov-Witten invariants of $X_\Sigma$, can be viewed as moduli spaces of {\it rational} maps from $\PP^1$
to $X_\Sigma$. The maps corresponding to boundary points
are allowed to have certain kind of {\it base-points}, depending on the toric data encoded by $\Sigma$. 
The compactifications make sense with $\PP^1$ replaced by any fixed smooth proper curve $C$ and are
projective schemes. When the target toric variety
is the projective space $\PP^n$, they were considered earlier by Drinfeld and named spaces of {\it quasimaps}.
We will call them toric quasimaps for general targets $X_\Sigma$.

Our goal in this paper is to extend this construction when the curve is not fixed in moduli, and also in the presence of markings. In other words, we seek
a {\it relative} compactification over the moduli space $\Mgk$ of stable curves. Furthermore, we want the compactification to be ``virtually smooth", that is,
to carry a perfect obstruction theory
(and hence a virtual class) which over the locus corresponding to honest maps coincides with the usual obstruction theory. 
For these to hold, it turns out that the appropriate condition to require is that the base points do not occur at markings or nodes. This is analogous
to the condition on {\it stable quotients} imposed in the paper \cite{MOP}, which was one of the major sources of inspiration for our work.
The resulting moduli stack will be called the {\it moduli space of stable toric quasimaps} and we prove in \S 3-5 that it is a proper Deligne-Mumford stack
of finite type with a canonical perfect obstruction theory.
One sees immediately that in the case of $\PP^n$, which is both a Grassmannian and a toric variety, the moduli space of \cite{MOP} and
ours coincide. 

Usually, when studying their deformation theory, 
moduli spaces of maps from nodal curves are viewed as stacks over the Artin stack $\fMgk$ of prestable curves.
The innovation in this paper is a change of perspective, based on the easy but important observation that obstructions
to deforming stable toric quasimaps are in fact obstructions to deforming {\it sections} of line bundles over the underlying curves, while
the deformations of the curves together with the line bundles on them are unobstructed. 
As a consequence, it is more natural to
view our moduli spaces as stacks over (products of) Picard stacks over $\fMgk$. The new point of view gives rise to a very
easy and transparent construction of the moduli spaces with their virtual classes: they are all presented as
zero loci of sections of vector bundles on smooth Deligne-Mumford stacks, see \S 3.2.
Furthermore, as we remark there, this new perspective applies equally well for stable maps
with target a toric variety, giving
a similar description of the moduli spaces and virtual classes\footnote{After our paper was posted on the arXiv, 
Yi Hu kindly informed us that a similar observation in the case of stable maps to $\PP^n$ was used earlier in his work \cite{HL} with Jun Li
to give a {\it local} description of the moduli space. We thank him for bringing this to our attention.}.   

Finally, we note that the change of perspective leads to generalizations. 
In a sequel \cite{CKM} to this paper, it will be shown that both the moduli of stable quotients and the moduli of stable toric quasimaps are particular instances
of a more general construction of virtually smooth compactifications (relative over $\Mgk$) of moduli spaces of maps from curves
to a class of GIT quotients.

The new moduli spaces are used to define a system of quasimap integrals, which are expected to be related to
(but in general different from) the Gromov-Witten invariants of the toric variety. These integrals give a new structure
of Cohomological Field Theory on $H^*(X_\Sigma)$. 
The last two sections of the paper touch briefly on the subject, but much of the study is left for future work.
In particular, the issue of defining a theory for complete intersections in toric varieties, as suggested in \cite{MOP},
fits better in the general framework of \cite{CKM} and will be treated there. Here we content ourselves to
showing how toric quasimaps in genus zero can be used to extend Givental's $I$-function to the ``big" parameter
space and to formulate a conjecture relating it to the big $J$-function from Gromov-Witten theory.
We give a short proof of this conjecture for (products of) projective spaces.

\subsection{Acknowledgments} During the preparation of the paper
we have benefited from conversations with Davesh Maulik and Rahul Pandharipande.
The material about the big $I$-function in \S \ref{big-I-fcn} is joint work with Rahul Pandharipande and we
are indebted to him for allowing us to include it in this paper. We also thank Chanzheng Li for
pointing out an inaccuracy in earlier version.
The research presented here, as well as the writing of the paper, were carried out
at the Korea Institute for Advanced Study in Summer 2009. Ciocan-Fontanine thanks KIAS for financial support, excellent
working conditions, and an inspiring research environment. 
Partial support for the research of Ciocan-Fontanine under the NSF grant DMS-0702871
and for the research of Kim under the grant KRF-2007-341-C00006 is gratefully acknowledged.

\section{Reminder on complex projective smooth toric varieties}

\subsection{Notations and general facts} Throughout the paper we work over the base field $\CC$.
Let $M\cong \ZZ^n$ be a $n$-dimensional lattice, let $N$ be its dual lattice and let $\Sigma\subset N_{\RR}$ be a complete nonsingular fan, with associated
$n$-dimensional smooth projective toric variety $X_\Sigma$. For every $0\leq i\leq n$ we denote by $\Sigma(i)$ the collection of $i$-dimensional
cones in $\Sigma$; we'll also write $\Sigma_{\mathrm {max}}$ for $\Sigma(n)$.

As usual, we identify a $1$-dimensional cone $\rho\in \Sigma(1)$ with its integral generator. Each such $\rho$ determines a Weil divisor $D_\rho$ on $X_\Sigma$.
We will write $\ZZ^{\Sigma(1)}$ for the free abelian group generated by the $D_\rho$'s. Let $l=|\Sigma(1)|$ and put $r=l-n$. 
There is an exact sequence
\begin{equation} 0 \ra M \ra \ZZ ^{\Sigma {(1)}} \ra \Pic (X_\Sigma ) \ra 0 \label{exact seq}\end{equation}
where the first map is given by $m\mapsto \sum _\rho \langle m,\rho \rangle D_\rho $,
while the second map is the obvious one. Since $\Pic (X_\Sigma )$ is torsion free, it 
is isomorphic to $\ZZ ^r$. The second map is
given by an integral $r\times l$ matrix $A=(a_{i\rho})$ of rank $r$, once we choose an integral basis $\{{\mathcal L}_1,\dots ,{\mathcal L}_r\}$ of 
${\mathrm {Pic}}(X_\Sigma)$. 
Applying ${\mathrm {Hom}}_\ZZ(-,\CC^*)$ we obtain the exact sequence
\begin{equation}\label{exact seq2}
1\ra \GG\ra (\CC^*)^{\Sigma(1)}\ra N\otimes \CC^*\ra 1.
\end{equation}
The group $\GG$ comes with an identification $\GG\cong (\CC^*)^r$ 
and under this identification, it
acts on the affine space $\CC^{\Sigma(1)}$ with weights given by the (transpose of the) matrix $A$. Explicitly,
if we denote by $z_\rho$ the coordinates in
$\CC^{\Sigma(1)}$, then ${\bf t}:=(t_1,\dots ,t_r)\in (\CC^*)^r$ acts by
$${\bf t}\cdot (z_{\rho_1},\dots ,z_{\rho_l})=({\bf t}^{{\bf a}_1}z_{\rho_1},\dots ,{\bf t}^{{\bf a}_l}z_{\rho_l}),$$
where ${\bf a}_j$ denotes the $j^{\mathrm {th}}$ column of the matrix $A$ and 
$${\bf t}^{{\bf a}_j}=t_1^{a_{1\rho_j}}t_2^{a_{2\rho_j}}\dots t_r^{a_{r\rho_j}}.$$

Associated to the fan $\Sigma$, there is an open $\GG$-invariant subset 
$$U(\Sigma) =\CC^{\Sigma(1)}\setminus Z(\Sigma)  $$
such that 
$$U(\Sigma)\lra U(\Sigma)/\GG=X_\Sigma$$
is a principal $\GG$-bundle, see \cite{Cox2}. If we denote by $z_\rho$ the coordinates in
$\CC^{\Sigma(1)}$, then the equations of $Z(\Sigma)$ are 
indexed by $\Sigma_{\mathrm {max}}$:
$$Z(\Sigma)=\{(z_\rho)\in \CC^{\Sigma(1)}\; |\; \prod_{\rho\not\subset\sigma}z_\rho=0\}.$$
It follows that $Z(\Sigma)$ is a union of linear subspaces in $\CC^{\Sigma(1)}$.
The toric divisors $D_\rho$ are the images of the coordinate hyperplanes $\{ z_\rho=0\}$ in $\CC^{\Sigma(1)}$ under the quotient
map.

This quotient construction of a smooth projective toric variety can also be interpreted in the usual
framework of GIT: linearizations of the trivial line bundle on $\CC^{\Sigma(1)}$ for which the semistable
and stable loci are equal and coincide with $U(\Sigma)$ correspond to the interior of the ample cone of $X_\Sigma$
(see e.g., \cite{Dol}, Chapter 12).

The following is well-known:

\begin{Lemma}\label{sigma basis}
Let $\sigma$ be a maximal cone of $\Sigma$. Then $\{ \cO(D_\rho)\;\; |\;\; \rho\not\subset\sigma\}$
is a $\ZZ$-basis of $\Pic(X_\Sigma)$. 
Furthermore, if $L$ is an ample line bundle on $X_\Sigma$, then the coefficients of $L$ in this
basis are positive.
\end{Lemma} 

\begin{proof} The first assertion follows from the regularity of the fan $\Sigma$, see \cite{Ewald}, Lemma VII.5.3 for a more
general statement. For the second assertion, see e.g. \cite{Kresch}, \S 2.2.
\end{proof}

\subsection{Maps to $X_\Sigma$\label{maps}} As shown by Cox in \cite{C}, the toric variety $X_\Sigma$ can be viewed as a fine moduli space. 
More precisely, for any scheme $Y$, Cox defines a {\it $\Sigma$-collection on $Y$} to be the following data:
\begin {itemize}
\item for each $\rho\in\Sigma (1)$, a line bundle $L_\rho$ on $Y$ {\it and} a section $u_\rho\in \Gamma(Y,L_\rho)$

\item  for each $m\in M$, a trivialization ${\phi}_m:\otimes_{\rho\in\Sigma(1)}L_\rho^{\otimes\langle \rho,m\rangle}\longrightarrow{\mathcal O}_Y$

subject to the conditions

\begin{equation}\label{comp}  (compatibility)\;\; \phi_{m}\otimes\phi_{m'}=\phi_{m+m'},\;\; \forall\;\; m,m'\in M
\end{equation}

\begin{equation}\label{nd} \begin{split} & (nondegeneracy)\;\; {\rm for\; each}\; y\in Y\; 
{\rm there\; exists\; a}\\ & {\rm maximal\; cone}\; \sigma\in\Sigma_{\rm max}\;
{\rm such\; that}\; u_\rho(y)\neq 0,\;\;\; \forall \rho\not\subset \sigma .\end{split}
\end{equation}
\end{itemize}

The functor ${\mathcal C}: \mathrm{Sch}\longrightarrow \mathrm{Sets}$ that $X_\Sigma$ represents is defined on objects by
$$ {\mathcal C} (Y) =\{ {\rm isomorphism}\; {\rm classes\; of}\; \Sigma-{\rm collections\; on}\; Y\},$$
while on morphisms it acts by pulling back $\Sigma$-collections.

The universal family on $X_\Sigma$ consists of the line bundles ${\mathcal O}(D_\rho)$ together with the canonical sections $v_\rho$
such that $D_\rho$ is the zero locus of $v_\rho$, and the isomorphisms
$${\mathcal O}(\sum_\rho\langle\rho,m\rangle D_\rho)\lra{\mathcal O}_{X_\Sigma}$$ induced for each $m\in M$ by the corresponding character $\chi^m$.

\section{Moduli of Stable Toric Quasimaps to $X_\Sigma$}

It follows from \S \ref{maps}
that the moduli space $Mor(Y,X_\Sigma)$ of maps from $Y$ to $X_\Sigma$ is identified with the moduli space of $\Sigma$-collections on $Y$, and one
may try to compactify it by allowing some degeneracies of the collections. When $Y=\PP^1$, 
this is precisely the strategy 
employed by Givental \cite{G} and by Morrison and Plesser \cite{MP} to obtain the ``toric compactifications" (or {\it linear sigma models}) they used to study the 
genus zero Mirror Conjecture. The degenerate maps are allowed to acquire some ``base points". Precisely, the nondegeneracy condition (\ref{nd}) is imposed
at all but finitely many points of the curve. 

In this section we extend the definition to (families of) nodal marked curves, introduce
the corresponding
moduli space, and show via a direct geometric construction
that it is a Deligne-Mumford stack of finite type with a perfect obstruction theory.
Properness
will be established in section \S \ref{properness}, while a discussion of the naturality of the obstruction theory is deferred to \S \ref{virt}.

\subsection{Definition and first properties}

Let $X_\Sigma$ be a projective smooth toric variety with a polarization 
${\mathcal O}_{X_\Sigma}(1) = \otimes_\rho  {\mathcal O}_{X_\Sigma}(D_\rho)^{\alpha_\rho}$.
Fix a genus $g\geq 0$, an integer $k\geq 0$,
and integers $d_\rho,\; \rho\in\Sigma(1)$.

\begin{Def} \label{stable qmap} A {\rm stable (pointed) toric quasimap} consists of the data 
$$( (C,p_1,\dots ,p_k), \{ L_\rho\}_{\rho\in \Sigma(1)}, \{ u_\rho\}_{\rho\in\Sigma(1)},  \{\phi _m\}_{m\in M}),$$ 
where

\begin{itemize}

\item $(C,p_1,\dots ,p_k)$ is a connected, at most nodal, projective curve of genus $g$ with $k$ distinct nonsingular marked points,

\item $L_\rho$ are line bundles on $C$, of degree $d_\rho$,

\item $\phi _m : \ot _\rho L_\rho ^{\langle\rho,  m\rangle} \ra \mathcal{O}_C$ are isomorphisms, satisfying the compatibility conditions (\ref{comp}),

\item $u_\rho \in \Gamma (C, L_\rho )$ are global sections,

\end{itemize}

satisfying

\begin{enumerate}

\item (nondegeneracy) there is a finite (possibly empty) set of nonsingular points $B\subset C$,
disjoint from the markings on $C$,
such that for every $y\in C\setminus B$,  there exists a maximal cone $\sigma \in \Sigma _{\mathrm{max}}$ 
with $u_\rho (y)\ne 0$, $\forall \rho \not\subset \sigma$,

\item (stability) $\omega _C(p_1+\dots +p_k) \ot {\mathcal L}^\epsilon $ is ample for every rational number $\epsilon > 0$, where
${\mathcal L}:=\otimes_{\rho\in\Sigma(1) } L_\rho^{\ot \alpha_\rho}$. 

\end{enumerate}

\end{Def}

\begin{Def} \label{isomorphism} An isomorphism between two stable toric quasimaps 
$$( (C,p_1,\dots ,p_k), \{ L_\rho\}_{\rho\in \Sigma(1)}, \{ u_\rho\}_{\rho\in\Sigma(1)},  \{\phi _m\}_{m\in M}),$$ and
$$( (C',p_1',\dots ,p_k'), \{ L_\rho '\}_{\rho\in \Sigma(1)}, \{ u_\rho '\}_{\rho\in\Sigma(1)},  \{\phi _m'\}_{m\in M}),$$
consists of an isomorphism $f:C\lra C'$ of the underlying curves, together with isomorphisms $\theta_\rho : L_\rho\lra f^*L_\rho '$,
such that the markings, the sections $u_\rho$, and the trivializations are preserved:
$$f(p_j)=p_j',\;\;\; \theta_\rho(u_\rho )=f^*(u_\rho '), \;\;\; \phi_m=f^*(\phi_m')\circ (\otimes_\rho \theta_\rho^{\langle\rho,m\rangle})
$$
\end{Def}

Stability is equivalent to
the following two conditions:
\begin{itemize}
\item The underlying curve $C$ cannot have any rational
components containing fewer than two special points (nodes or markings).
\item On every rational component with exactly two special points, or on every elliptic component
with no special points, the line bundle ${\mathcal L}$ must have positive degree.
\end{itemize}
In particular, stability imposes the inequality $2g-2+k\geq 0$, which we will assume for the rest of this section.

The points of the subset  $B\subset C$ introduced in the
nondegeneracy condition of Definition \ref{stable qmap} will be called the {\it base-points} of the quasimap. By definition,
a stable toric quasimap has at most finitely many base-points, all nonsingular and nonmarked.

As defined above, the notion of stable toric quasimap depends a priori both on the fan $\Sigma$ used to present $X_\Sigma$ as a toric variety, and 
on the chosen polarization ${\mathcal O}_{X_\Sigma}(1)$. 

Given a stable toric quasimap,
there is an induced homomorphism from $\Pic (X_\Sigma)$ to $\Pic (C)$, given by ${\mathcal O}_{X_\Sigma}(D_\rho)\mapsto L_\rho$, which in turn defines
via Poincar\' e duality an integral curve class $\beta\in H_2(X_\Sigma,\ZZ)$. Precisely, $\beta$ is determined by
$$\beta \cdot D_\rho=d_\rho,\;\; \rho\in\Sigma(1).$$ 
Furthermore, the same argument assigns a
homology class $\beta_{C'}$ (determined by $\beta_{C'} \cdot D_\rho= \deg( L_\rho|_{C'}$)
to each irreducible component $C'$ of $C$, with $\beta=\sum_{C'}\beta_{C'}$.

We will say that the toric quasimap 
is {\it of class $\beta$}.

\begin{Lemma}\label{indep} The homology classes $\beta_{C'}$ 
are all effective. In particular, if a toric quasimap is stable for
one polarization ${\mathcal O}_{X_\Sigma}(1)$, then it is so for any other polarization. 
\end{Lemma}

\begin{proof} We need to show that 
$$\int_{\beta_{C'}}c_1(P)\geq 0$$
for any ample line bundle $P$ on $X_\Sigma$, and equality holds if and only if $\beta_{C'}=0$.
 
By the finiteness of the set of base points, there exists $p\in C'$ and a maximal cone $\sigma\in \Sigma(n)$ such
that the sections $u_\rho, \rho\not\subset\sigma$ are nonvanishing at $p$. Let $P$ be any ample line bundle.
By Lemma \ref{sigma basis}, we can 
write 
$$P=\ot_{\rho\not\subset\sigma}\cO(D_\rho)^{\ot b_\rho}.$$
with all $b_\rho >0$. By the definition of $\beta_{C'}$ we have
$$\int_{\beta_{C'}}c_1(P)=\sum_{\rho\not\subset\sigma} b_\rho{\mathrm {deg}}(L_\rho|_{C'}).$$
Since $\prod_{\rho\not\subset\sigma}u_\rho^{b_\rho}$ is a global section of $\ot _{\rho\not\subset\sigma}L_\rho^{\ot b_\rho} |_{C'}$ which is
nonvanishing at $p$, the degree of this line bundle is nonnegative. Moreover, the degree is zero if and only if the section is
constant. This implies that all $u_\rho$'s are constant and hence $\beta_{C'}=0$.

\end{proof}

\begin{Rmk}\label{effective}Note that the above proof only uses part of nondegeneracy property  $(1)$ in Definition \ref{stable qmap}, namely
that it holds at the generic point of each irreducible component of $C$ (and does not use stability). This will be used later.
\end{Rmk}
\begin{Cor}\label{boundedness-curve} For fixed, $g$, $k$, and $\beta$, the number of components of the underlying
curve of a stable toric quasimap is bounded.
\end{Cor}

\begin{proof} We only need to bound the number of {\it unstable} rational components with no markings on them. Such components must
have exactly two nodes by stability
and the proof of Lemma \ref{indep} shows that their homology classes are effective and nonzero.
Hence their number is bounded.

\end{proof}

As usual, the stability condition is imposed in order to rule out infinitesimal automorphisms:

\begin{Lemma}\label{autom} The automorphism group of a stable toric quasimap is
finite and reduced.
\end{Lemma}

\begin{proof} We only need to consider rational components $C'$ of $C$ with exactly two special points. As in the 
proof of Lemma \ref{indep}, by (generic) nondegeneracy
$\cL|_{C'}$ has a nonzero global section, while by stability $\cL|_{C'}$ is nontrivial. Since an automorphism of the quasimap
must preserve this section, the lemma follows.

\end{proof}

\begin{Def}\label{family}
A family of genus $g$ stable toric quasimaps to $X_\Sigma$ of class $\beta$ over a scheme $S$ consists of the data
$$(\pi:{\mathcal C}\ra S, \{ p_i:S\ra {\mathcal C} \}_{i=1,\dots ,k}, \{L_\rho\}, \{u_\rho\}, \{\phi_m\})
$$
where 
\begin{itemize}
\item ${\pi:\mathcal C}\ra S$ is a flat, projective morphism of relative dimension one,
\item $p_i$, $i=1,\dots ,k$ are sections of $\pi$,
\item $L_\rho$, $\rho\in\Sigma(1)$ are line bundles on ${\mathcal C}$,
\item $u_\rho\in\Gamma({\mathcal C},L_\rho)$ are global sections,
\item $\phi _m : \ot _\rho L_\rho ^{\langle\rho,  m\rangle} \ra \mathcal{O}_C$ are isomorphisms, satisfying the compatibility conditions (\ref{comp}),
\end{itemize}
such that the restriction of the data to every geometric fiber ${\mathcal C}_s$ of $\pi$ is a stable $k$-pointed toric quasimap of genus $g$ and class $\beta$.

An isomorphism between two  such families $({\mathcal{C}}\ra S,...)$ and $({\mathcal{C}'}\ra S,...)$ consists of an isomorphism of
$S$-schemes $f:{\mathcal C}\lra {\mathcal C}' $ and a collection of isomorphisms $L_\rho\lra f^*{L'}_\rho$ which preserve the markings,
the trivializations, and the sections.

\end{Def}

We therefore obtain a stack $\mathcal{T}_{g,k}(X_\Sigma,\beta)$ whose objects are isomorphism classes
of families of stable toric quasimaps to $X_\Sigma$,
with fixed discrete parameters $g,k$, and $\beta$.

It will be convenient for some of the arguments in the rest of the paper to consider an equivalent description of the data
defining a toric quasimap (furthermore, this description is the
one that fits in the general framework of quasimaps to GIT quotients which will be treated elsewhere). 
 For this, recall that we have chosen an integral basis $\{{\mathcal L}_1,\dots ,{\mathcal L}_r\}$ of ${\mathrm{Pic}}(X_\Sigma)$ and denoted by $A=(a_{i\rho})$
 the matrix of the natural map $\ZZ^{\Sigma(1)}\ra{\mathrm{Pic}}(X_{\Sigma})$, so that 
 $${\mathcal O}(D_\rho)=\bigotimes_{i=1}^r{\mathcal L}_i^{\otimes a_{i\rho}}.$$
 The exact sequence (\ref{exact seq}) then shows the following.
 
 \begin{Lemma} \label{qmap2}
 Let $\beta\in H_2(X_\Sigma,\ZZ)$ be an effective curve class. Isomorphism classes of stable toric quasimaps of class $\beta$ are given 
 by isomorphism classes of the data 
 $$( (C,p_1,\dots ,p_k), \{ {\mathcal P}_i\; |\; i=1,\dots, r\}, \{ u_\rho\}_{\rho\in\Sigma(1)}),$$ 
where

\begin{itemize}

\item $(C,p_1,\dots ,p_k)$ is a connected, at most nodal, $k$-pointed projective curve of genus $g$, with $p_i$ nonsingular distinct points,

\item ${\mathcal P}_i$ are line bundles on $C$, of degrees $f_i:=\int_\beta c_1({\mathcal L}_i)$,

\item $u_\rho \in \Gamma (C, L_\rho )$ are global sections of the line bundles $$L_\rho:=\bigotimes_{i=1}^r{\mathcal P}_i^{\otimes a_{i\rho}}, $$
  \end{itemize}
  subject to the nondegeneracy and stability conditions of Definition \ref{stable qmap}. Here the line bundle $\cL$
  in the stability condition is
  $$\cL=\otimes_{\rho\in\Sigma(1)}(\otimes_{j=1}^r \cP_j^{\otimes a_{j\rho}})^{\otimes \alpha_\rho} ,$$
  and isomorphisms of the data are isomorphisms of the underlying curve preserving the markings, the line bundles ${\mathcal P}_i$, and
  the sections $u_\rho$.
  \end{Lemma}
 
 Similarly, we have a description for isomorphism classes of families equivalent to that in Definition \ref{family}. More precisely, given a base scheme $S$,
 the groupoid whose objects are families of stable toric quasimaps over $S$ and whose morphisms are isomorphisms of such families is equivalent
 to the groupoid whose objects are families over $S$ of data as in Lemma \ref{qmap2} and whose morphisms are isomorphisms of the data. It follows that
 we may use the data in the lemma to define the moduli stack $\mathcal{T}_{g,k}(X_\Sigma,\beta)$.
 
 \begin{Rmk} The alternative definition makes it obvious that when $X_\Sigma=\PP^n$ with 
 the usual polarization $\cO_{\PP^n}(1)$, the functor 
 ``stable toric quasimaps to $\PP^n$of degree $d$" and the functor ``stable quotients of degree $d$ to $\PP^n$" of \cite{MOP} coincide.
 Indeed, a quotient
 $$0\lra S\lra\cO_{C}^{\oplus n+1}\lra Q\lra 0$$
 of rank $n$ and degree $d$ with no torsion at nodes and markings is equivalent to $n+1$ sections of $S^\vee$ such that the map
 $\cO_C^{\oplus n+1}\lra S^\vee$ is generically surjective, and surjective at nodes and markings. 
 This is precisely our nondegeneracy condition. The stability conditions also match.
 \end{Rmk}
 
 Corollary \ref{boundedness-curve} establishes boundedness for the underlying curves in families
 of stable toric quasimaps. We will also need a boundedness result
 for the line bundles $\cP_j$.
 
 \begin{Lemma} Let 
 $$( \cC\ra S, \{p_i:S\ra\cC\; |\; i=1,\dots ,k\}, \{ {\mathcal P}_j\; |\; j=1,\dots, r\}, \{ u_\rho\}_{\rho\in\Sigma(1)})$$
 be a family of $k$-pointed stable toric quasimaps of genus $g$ and class $\beta$, parametrized by a 
 scheme $S$. There exists a positive integer $c=c(\beta)$, depending only on $\beta$, such that for every
 $1\leq j\leq r$, every geometric point $s$ of $S$, and every component $C'$ of $\cC_s$ we have
 $$\left |{\mathrm {deg}}(\cP_j|_{C'})\right |<c.$$
  
 \end{Lemma}
 
 \begin{proof} Pick a component $C'$ of $\cC_s$. By the generic nondegeneracy condition, there
 is a maximal cone $\sigma\in\Sigma(n)$ such that 
 each $u_\rho$, $\rho\not\subset\sigma$ is a nonzero section of the corresponding $L_\rho$ on $C'$.
 Hence $L_\rho$ has nonnegative degree on $C'$. By Lemma \ref{sigma basis} we can write the polarization
 $\cO_{X_\Sigma}(1)$ as $\ot_{\rho\not\subset\sigma}\cO(D_\rho)^{\ot \alpha_\rho}$ with $\alpha_\rho >0$.
 From this and Lemma \ref{indep} we get that for every $\rho\not\subset\sigma$,
 $${\mathrm {deg}}(L_\rho|_{C'})\leq\int_\beta c_1(\cO_{X_\Sigma}(1)).$$
 Since each $\cP_j$ is expressed uniquely as a linear combination of $L_\rho$, $\rho\not\subset\sigma$, the lemma follows. 
 
 \end{proof}
 
 \begin{Cor}\label{boundedness}  For a family as above, let $\cO_\cC(1)$ be an $S$-relative
 ample line bundle. There exists an integer $m>\!>0$, depending 
 only on $g$, $k$, and $\beta$, such that for every geometric point $s\in S$ and every $1\leq j\leq r$, the line
 bundle $\cP_j(m)$ is globally generated on $\cC_s$, with $$H^1(\cC_s,\cP_j(m))=0.$$ Similarly, there exists
 $m_1>\!>m$ such that the same conclusions hold for each of the line bundles $L_\rho$.
 
 \end{Cor}

\subsection{A construction of $\cTX$}  \label{new-construction} The goal of this subsection is to prove 
\begin{Thm} \label{construction} The stack $\mathcal{T}_{g,k}(X_\Sigma,\beta)$ is a Deligne-Mumford stack, of finite type over $\CC$.
Moreover, it admits a perfect obstruction theory.
\end{Thm}

It is possible (and relatively straightforward) to prove the theorem via the usual
strategy: first present $\mathcal{T}_{g,k}(X_\Sigma,\beta)$ as the global stack-quotient of a scheme of finite type by a complex
algebraic group, then apply Lemma \ref{autom} together with Theorem. 4.21 of \cite{DM}
(see also \cite{Edidin}, Cor. 2.2) to conclude the theorem. The quotient construction is similar
to the one in \cite{MOP}. One finds a quasiprojective scheme parametrizing stable toric quasimaps
$$( (C,p_1,\dots ,p_k), \{ {\mathcal P}_i|i=1,\dots, r\}, \{ u_\rho\}_{\rho\in\Sigma(1)})$$ of genus $g$ and class $\beta$ together
with some additional structure which rigidifies them, and then one divides by the isomorphisms of the additional structure.

We choose to give a more direct proof,
by realizing $\cTX$ as a closed substack in a smooth Deligne-Mumford
stack $\cE$, given by explicit equations. More precisely, we construct $\cE$ together with a vector bundle
$\cV$ on it and a natural section $s$ of $\cV$, such that
$$Z(s)=\cTX.$$
Here $Z(s)$ is the stack-theoretic zero locus of $s$. The advantage of this method is that it provides
immediately a perfect obstruction theory for $\cTX$
whose associated virtual class  $[\cTX]^{\mathrm{vir}}$ is simply the refined top Chern class of $\cV$.
A more detailed discussion of the obstruction theory is given in \S\ref{virt} below.

\begin{proof}
Let $\fMgk$ be the stack of prestable $k$-pointed curves of genus $g$ and let $\mathfrak{P}ic\ra\fMgk$ be the
the relative Picard stack. We denote by
$$\mathfrak{P}ic^r{\stackrel{\phi}\lra}\fMgk$$ the $r$-fold fibered product of $\mathfrak{P}ic$ over $\fMgk$.
It is well-known that both stacks are smooth Artin stacks of infinite type and the morphism $\phi$ is smooth.

Fix an effective class $\beta\in H_2(X_\Sigma,\ZZ)$. There are morphisms of stacks fitting in the commutative diagram
\begin{equation}\label{diagram}\begin{array}{ccc}  \cTX   &= &\cTX \\  \\
 \mu\downarrow & &\downarrow \nu \\  \\
\mathfrak{P}ic^r&\stackrel\phi\lra &\fMgk 
\end{array}\end{equation}
with $\mu$ forgetting the sections $u_\rho$ and $\nu$ forgetting both the sections and the line bundles $\cP_i$.

By Corollary \ref{boundedness-curve} there exists an open and closed substack of finite type
$\fS\subset\fMgk$ such that the map $\nu$ factors through $\fS$. We may take $\fS$ such that its points
parametrize prestable curves with no rational tails.

Recall that specifying the class $\beta$ of a stable toric quasimap
 $$( (C,p_1,\dots ,p_k), \{ {\mathcal P}_i\; |\; i=1,\dots, r\}, \{ u_\rho\}_{\rho\in\Sigma(1)})$$ 
 is equivalent to specifying the degrees $f_i$ of the line bundles $\cP_i$.  
 Let
 $$\pi:\cC\lra\fS$$ 
 be the universal curve over $\fS$.

 On $\cC\times_\fS \mathfrak{P}ic^r$ we have universal line bundles 
$\cQ_j$, $j=1,\dots ,r$, 
restricting to $Q_j$ on the fiber over a point $((C,p_1,\dots,p_k), Q_1,\dots ,Q_r)\in\mathfrak{P}ic^r$.
Therefore, we also have universal line bundles
$$\cQ_\rho:=\ot_j\cQ_j^{\otimes a_{j\rho}}.$$
Let $\cO_\cC(1)$ be a $\pi$-relative ample bundle. 
Now fix an integer $m_1>\!\!>0$ (depending only on $g$, $k$, and $\beta$) as in Corollary \ref{boundedness}. 
In addition, we also fix an injection
\begin{equation}\label{section}
0\lra\cO_\cC\lra\cO_\cC(m_1).
\end{equation}

We denote by
 $$\mathfrak{P}ic_{\beta}^r{\stackrel{\phi_\beta}\lra}\fS$$
 the substack of $\mathfrak{P}ic^r$ obtained by imposing the conditions 
 
 $(i)$ the degree of $Q_j$ on $C$ is equal to $f_j$ for all $j=1,\dots, r$,
 
 $(ii)$ $H^1(C,Q_\rho(m_1))=0$ for all $\rho\in\Sigma(1)$,

 Since each of these is an open condition on the base of a family,  $\mathfrak{P}ic^r_\beta$ is an open substack. 
 Hence it is a smooth Artin stack of finite type.
 The map $\mu$ factors through $\mathfrak{P}ic^r_\beta$.

By condition $(ii)$, for each $\rho\in\Sigma(1)$ the sheaf
$$\pi_*\cQ_\rho(m_1) := \pi_*(\cQ_\rho\ot q_\cC^*\cO_\cC(1)^{m_1})$$
is a vector bundle on $\mathfrak{P}ic^r_\beta$ for every $\rho$. Here $q_\cC$ and $\pi=\pi_{\mathfrak{P}ic}$ are the projections
from $\cC\times_\fS \mathfrak{P}ic^r_\beta$ to the two factors.

Let $$p:\mathfrak{X}\lra\mathfrak{P}ic_\beta^r$$ denote the total space of the bundle $\oplus_\rho (\pi_*\cQ_\rho(m_1))$.
It is a smooth Artin stack of finite type and $p$ is a representable smooth morphism.
Points in the fiber of $p$ over a point
$$((C,p_1,\dots,p_k), Q_1,\dots ,Q_r)\in\mathfrak{P}ic^r_\beta$$
correspond to collections of sections
$$\{v_\rho\in H^0(C, Q_\rho (m_1))\;\; |\;\; \rho\in\Sigma(1)\}.$$

We let $$\cX'\subset \mathfrak{X} $$ be the substack determined by the requirement

$(iii)$ the {\it generic} nondegeneracy condition holds for the sections $v_\rho$

$(iv)$ the stability condition holds, i.e., the line bundle
$$ \omega_C(\sum_{i=1}^k p_i)\otimes \left(\otimes_\rho(\otimes_j Q_j^{\otimes a_{j\rho}})^{\otimes \alpha_\rho}\right)^\epsilon$$
is ample on $C$ for every $\epsilon >0$.

Since these are both open conditions on the base of a family, $\cX'$ is an open substack. Furthermore, the proof of Lemma \ref{autom}
implies that its points have finite automorphism groups. We conclude that $\cX'$ is Deligne-Mumford and smooth.

By construction, the stack $\cX'$ parametrizes data
$$( (C,p_1,\dots ,p_k), \{ {Q}_i\; |\; i=1,\dots, r\}, \{ v_\rho\}_{\rho\in\Sigma(1)}),$$
satisfying the conditions of Lemma \ref{qmap2}, except that the $v_\rho$'s are sections of the 
twisted line bundles $Q_\rho(m_1)$, and that they may have base-points at nodes or markings.

For each $\rho\in\Sigma(1)$, the injection of sheaves (\ref{section}) induces an identification of
$H^0(C,Q_\rho)$ with a vector subspace of $H^0(C,Q_\rho(m_1))$.
We define a closed substack $\cX''\subset\cX'$
by imposing the condition
\vskip.1in
$(v)$ for each $\rho\in\Sigma(1)$, the section $v_\rho$ lies in the subspace $H^0(C,Q_\rho)$ of $H^0(C,Q_\rho(m_1))$.

\vskip.1in
Let $\cV_\rho$ be the sheaf on $\cC\times_\fS \mathfrak{P}ic^r_\beta$ defined by the exact sequence

$$0\lra\cQ_\rho\lra\cQ_\rho(m_1)\lra\cV_\rho\lra 0.
$$
Then $\pi_*(\cV_\rho)$ is a vector bundle on $\mathfrak{P}ic^r_\beta$. The vector bundle

\begin{equation}\label{bundle}
\cV:= p^*\left (\oplus_{\rho\in\Sigma(1)}\pi_*(\cV_\rho)\right )\end{equation}
on $\cX'$ comes with a tautological section $s$, induced by the maps 
$$H^0(\cC,\cQ_\rho(m_1))\lra H^0(\cC,\cV_\rho),$$
whose vanishing corresponds exactly to imposing the condition $(v)$ above. Hence $\cX''$ is identified with the closed
substack $Z(s)$ in $\cX'$.

Finally, the stack $\cTX$ is then singled out as the open substack in $\cX''$ obtained by requiring
\vskip.1in
$(vi)$ there are no base points at nodes or markings.
\vskip.1in
We note that in principle the order in which the last two conditions $(v)$ and $(vi)$ are imposed could be reversed. Namely, there
is an open substack $\cE\subset \cX'$ (which is necessarily Deligne-Mumford and smooth) such that $\cTX$ is cut out in $\cE$ by
imposing condition $(v)$, i.e., as the zero locus $Z(s)$ of the tautological section of the bundle (\ref{bundle}) on $\cE$. However,
we do not know an explicit modification of the condition $(vi)$ that would describe $\cE$. This substack $\cE$ will be
used in subsection 5.3.

\end{proof}

\begin{Rmk}  
$(a)$ An almost identical construction can be given for the moduli stack $\Mgk(X_\Sigma,\beta)$ of
stable maps to $X_\Sigma$, with its usual obstruction theory. The only modifications needed are in 
conditions $(iv)$ and $(vi)$. Namely, the stability condition imposed is that 
$$\omega_C(\sum_ip_i)\otimes \left(\otimes_\rho(\otimes_j Q_j^{\otimes a_{j\rho}})^{\otimes \alpha_\rho}\right)^{\otimes 3}$$ 
is ample, while the nondegeneracy condition imposed is that the sections have {\it no} base
points.
Of course, the underlying curve of a stable map is allowed to have rational tails, so the substack of finite
type $\fS\subset\fMgk$ will be different.
We leave the details to the reader.

$(b)$ For a fixed nonsingular domain curve $C$ with no markings, the realization of the moduli space
as the zero locus of a section of a bundle on a toric fibration over the product 
of Jacobians $\mathrm{Jac}(C)^r$ was observed in \cite{OT}. 
In our construction the toric fibration is hidden by the use of Artin stacks and reemerges only after passing
to $(\CC^*)^r$ covers of $\cE$ and $\mathfrak{P}ic^r_\beta$.

$(c)$ By Lemma \ref{indep}, the moduli stack $\mathcal{T}_{g,k}(X_\Sigma,\beta)$ does not depend on the
choice of the polarization ${\mathcal O}_{X_\Sigma}(1)$. However, note we do not claim that it is
also independent on $\Sigma$ (or, in other words, on the presentation of $X$ as
a GIT quotient) and in fact we suspect that such independence does not hold, even though we do not have a 
specific example. A similar situation occurs for the case of stable quotients studied in \cite{MOP}:
the Grassmannians $G(1,n)$ and $G(n-1,n)$ are both isomorphic to the $(n-1)$-dimensional
projective space, but the corresponding moduli stacks of stable quotients are very different.
On the other hand, see Conjecture \ref{I=J} for a statement about the independence of the quasimap {\it invariants}
when the toric variety is assumed to be Fano. \end{Rmk}

\section{Properness}\label{properness}

\begin{Thm} The stack $\cTX$ is proper.
\end{Thm}
We will split the argument into two parts.

\subsection{Separatedness} 
The argument is similar to the one in \cite{MOP}.

Let $R$ be a DVR over $\CC$, with quotient field K. Let $\Delta =\Spec(R)$ and let $0\in\Spec(R)$ be the closed point;
put $\Delta^0=\Delta\setminus\{ 0\}=\Spec(K)$. 

Assume we have two flat families of $k$-pointed stable toric quasimaps 
$$({\mathcal C}_i\ra \Delta, \{p_1^i,\dots, p_k^i:\Delta\ra\cC_i\},\{\cP_1^i,\dots ,\cP_r^i\},\{u_\rho^i\; |\; \rho\in\Sigma(1)\}),\; i=1,2 $$ 
which are isomorphic over $\Delta^0$.

By semistable reduction, we can find a third family ${\mathcal C}$ of prestable $k$-pointed curves over $\Delta$ which dominates both ${\mathcal C}_i$'s
and preserves the sections.
We may assume that the central fiber $C_0$ of $\cC$ has automorphism group
at most 1-dimensional, since the central fibers $C_{0,i}$ of $\cC_i$ both have the same property. 
This means that $\cC\lra \cC_i$ is obtained by blowing up only nodes in the central fiber $C_{0,i}$.

Now pull back the bundles $\cP_j^i$ and the
sections $u_\rho^i$ on $\cC_i$ to $\cC$.  
Since the nondegeneracy condition holds at the nodes of $C_{0,i}$,
we obtain families over $\Delta$ of {\it prestable} toric quasimaps, i.e., families of data as in Lemma \ref{qmap2}, satisfying the nondegeneracy condition,
but possibly not the
stability condition on some components of the central fiber $C_0$. 

Over $\Delta^0$ we have an isomorphism between the two families.
Let ${\mathcal{B}}_1,{\mathcal{B}}_2\subset\cC$ be the base loci of the two families. Their union intersects $C_0$ in a
finite subset $B_0$ of nonsingular points.
Via Cox's description from $\S\ref{maps}$, we can think of either family as giving a regular map from $\cC\setminus({\mathcal{B}}_1\cup{\mathcal{B}}_2)$
to $X_\Sigma$. These maps agree over $\Delta^0$, therefore they must be the same.
It follows that the isomorphism between the two families extends to an isomorphism on $\cC\setminus B_0$. 
Since $B_0$ is a finite set of smooth points on a surface, the isomorphism extends to all of $\cC$.

Finally, there cannot be any component of $C_0$ which is contracted to a node by, say, the map
$\pi _1: \cC\ra \cC_1$ but not by $\pi _2: \cC\ra \cC_2$.
Indeed, over such a component $\pi _1^*\cL^1$ is 
trivial, while by stability  $\pi _2^*\cL^2$ is not trivial (here $\cL^1$ and $\cL^2$ are the 
line bundles measuring stability on the families $\cC_1$ and $\cC_2$, respectively). 
It follows that the initial families $(\cC_i\ra\Delta,\dots )$ are
isomorphic.
Hence $\cT_{g,k}(X_\Sigma, \beta)$ is separated by the valuative criterion.

\subsection{Completeness}

Let $R$ be a DVR over $\CC$ and let $K$ be its field of fractions. It suffices to show that a given
family of stable toric quasimaps over $\Delta^0:=\Spec(K)$ extends to a 
family of stable toric quasimaps over $\Delta:=\Spec( R)$.

First we introduce some terminology: we say that data
$$((C,p_1,\dots, p_k), \{\cP_j\}, \{u_\rho\})$$
as in Lemma \ref{qmap2} forms a {\it quasistable toric quasimap to $X_\Sigma$} if it satisfies all requirements
of a stable toric quasimap except that base-points are allowed at nodes and markings. Note that
a quasistable toric quasimap has a well-defined effective homology class
$\beta$ by Remark \ref{effective}.

The following lemma, whose proof 
is given at the end of this section, will be a key part of the argument.

\begin{Lemma}\label{lemmaB} Let $\cC\lra \Delta$ be a family of prestable $k$-pointed curves of genus $g$, 
with each fiber having automorphism group of dimension at most 1. Let
$\cP_1,\dots ,\cP_r$ be line bundles on $\cC$ and for each $\rho\in\Sigma(1)$, let $u_\rho$ be
a rational section of 
$$L_\rho= \bigotimes_{i=1}^r{\mathcal P}_i^{\otimes a_{i\rho}}$$
which is regular on $\cC|_{\Delta^0}$. Assume that this data forms a family of stable toric quasimaps of class $\beta$
over $\Delta^0$. Then there are unique modifications $\widetilde{L_\rho}$ of $L_\rho$ at points on the central fiber $C_0$ 
such that $u_\rho$ extend to regular sections of $\widetilde{L_\rho}$, and such that their base locus on the 
central fiber is a finite set. Furthermore, there is a contraction 
$$\begin{array}{ccc} \cC&\stackrel f\longrightarrow&\cC'\\ 
\;\; \;\;\;\;\; \searrow & &\swarrow \;\;\;\;\;\; \\
& \Delta &
\end{array} $$
which is an isomorphism over $\Delta^0$,
such that for each $\rho\in\Sigma(1)$ the push-forward
$f_*\widetilde{L_\rho}$ is a line bundle, the section $u_\rho$ descends to a section of $f_*\widetilde{L_\rho}$, and this data
is a family of
quasistable toric quasimaps of class $\beta$ over $\Delta$.
\end{Lemma} 

Now consider a family
$$(\pi^0:\cC^0\ra\Delta^0,\; p_1,\dots ,p_k:\delta^0\ra\cC^0,\; \cP_1^0,\dots,\cP_r^0,\; \{u_\rho\}_{\rho\in\Sigma(1)})$$
of stable toric quasimaps. 
As in \cite{MOP}, \S 6.3, we may assume that $\pi^0:\cC^0\ra\Delta^0$ has nonsingular, irreducible fibers and that it extends to
$$\pi:\cC\lra\Delta$$
such that $\cC$ is a nonsingular surface over $\Spec(\CC)$ and the special fiber $C_0=\pi^{-1}(0)$ is a prestable pointed
curve with $\dim \mathrm{Aut}(C_0)\leq 1$. We fix a $\pi$-relative ample line bundle $\cO_{\cC}(1)$.

By Corollary \ref{boundedness} we can find $m>>0$ such that 
$\pi_*((\cP_j^0)^\vee(m))$ is a vector bundle on $\Delta^0$ (hence trivial) for all $j$.  It follows that there are
surjections
$$\cO_{\cC^0}(-m)^{\oplus N}\twoheadrightarrow (\cP_j^0)^\vee,$$
and by dualizing we get
$$0\lra\cP_j^0\lra\cO_{\cC^0}(m)^{\oplus N}\lra Q_j^0\lra 0$$
The quotient $Q_j^0$ has Hilbert polynomial $h_j$ depending only on the degree of $\cP_j^0$. By properness of the 
relative quot functor $\cQ uot_{h_j}(\cO(m)^{\oplus N})$, we obtain an extension 
$$0\lra\cP_j\lra\cO_\cC(m)^{\oplus N}\lra Q_j\lra 0$$
of the above exact sequence to
$\cC$ such that $Q_j$ is $\pi$-flat. 
After possibly replacing $\cP_j$ by its double dual, we may assume that $\cP_j$ is a line
bundle.

We have therefore constructed a family to which Lemma \ref{lemmaB} can be applied 
to obtain a quasistable toric quasimap 
$$(\cC'\lra\Delta,\dots ). $$

To finish the proof we need to remove the possible base-points at nodes or markings in the
central fiber.
Let $q\in C'_0$ be a node or marking
which is a base point and let $\widetilde{\cC'}\lra\cC'$ be the blow-up of $\cC'$ at $q$.  
The central fiber has an additional component, which is a $\PP^1$ with two special
points on it. The pull-back to $\widetilde{\cC'}$ of the quasistable toric quasimap on $\cC'$ satisfies the assumptions of
Lemma \ref{lemmaB}. Hence we can find a modification (along the new component $\PP^1$ only) such that
the resulting family has finite base locus on the central fiber. By Remark \ref{effective}, the new component
has a homology class $\beta_{\PP^1}$ which is effective.
We claim that $\beta_{\PP^1}\neq 0$. Indeed, if not, then all $L_\rho$'s must be trivial on this component, with
the sections $u_\rho$ constant and satisfying the nondegeneracy condition at every point. But this is impossible,
since $q$ was assumed to be a base point.

We conclude that the modification gives a quasistable toric quasimap on  $\widetilde{\cC'}$. After 
replacing $\cC'$ by $\widetilde{\cC'}$, the procedure can now be repeated. Since the number 
of components with nonzero homology class of a quasistable toric quasimap of class $\beta$
is bounded, the process must produce after finitely
many steps a family of {\it stable} toric quasimaps. 
Therefore, to finish the argument we only need to give the proof of Lemma 4.2.1.

\vskip .3in

{\it Proof of Lemma \ref{lemmaB}}:
The key idea is the use of the properness of $X_\Sigma$.
Let $\mathcal{B} ^0$ be the base locus on $\cC |_{\Delta _0}$ of the quasimap. Since $\cC$ is normal and 
$X_\Sigma$ is proper, the induced map
from $\cC |_{\Delta _0} \setminus \mathcal{B}_0$ to $X_\Sigma$ can be extended 
to a map to $X_\Sigma$ from $\cC \setminus \overline{
\mathcal{B} ^0}$ except possibly on a finite subset 
$$B_0\subset C_0\setminus \overline{\mathcal{B} ^0},$$ 
where $\overline{\mathcal{B} ^0}$ is the closure
of $\mathcal{B}^0$ in $\cC$. 

Via pull-back by this map, we obtain quasimap data
on $\cC \setminus (\overline{\mathcal{B} ^0}\cup B_0)$ 
(in the obvious sense, generalizing the notion of quasimaps to families of possibly  
nonproper nodal curves). By construction, it agrees with the quasimap data on $\cC |_{\Delta _0}$, 
so by gluing we get a quasimap 
$$((\cC \setminus B'_0,p_1\dots, p_k), \{ \widetilde{\cP_i} \}, \{\widetilde{u_\rho}\}),\;\;\; B'_0=B_0\cup(\overline{\mathcal{B} ^0}\cap C_0)$$ 
on $\cC \setminus B'_0$.
Since $B'_0$ consists of
finitely many normal points, the line bundles $\widetilde{\cP_i}$ are defined globally on $\cC$ and by Hartog's Theorem
the sections $\widetilde{u_\rho}$ also extend to regular sections of the 
line bundles $\widetilde{L_\rho}$ on $\cC$. The nondegeneracy condition holds at the generic points
of components of the special fiber by construction.
The uniqueness is clear. 

The stability condition may fail on some components of the special fiber. However, by Remark
\ref{effective}, components of the central fiber have a well-defined class $\beta_C$ which is effective. 
A component is unstable if and only if it is a $\PP^1$
with two special points, all line bundles are trivial on it and the sections are constant.  
Such components appear in chains of rational curves in $C_0$.

Now, first of all, there is a family of semistable curves
$\cC '\lra \Delta$, with normal total space, which is obtained from $\cC$ by contracting all these rational components:
$$\begin{array}{ccc} \cC&\stackrel f\longrightarrow&\cC'\\ 
\;\; \;\;\;\;\; \searrow & &\swarrow \;\;\;\;\;\; \\
& \Delta &
\end{array} $$

Let $D$ a chain of $\PP^1$'s  contracted by $f$
and let $p=f(D)\in\cC'$. There is a small enough analytic (or \'{e}tale) neighborhood $p\in U'\subset\cC'$
such that on $U:=f^{-1}(U')$ the toric quasimap
defines an honest map
$$g:U\lra X_\Sigma$$
which is constant on the chain $D$. Hence $g$ factors through the contraction $U\lra U'$,
at least as a continuous map $g':U'\lra X_\Sigma$.
Since $U$ is normal at $p$, Hartog's Theorem implies that $g'$ is
a regular map. Now the pull-back of toric data via these $g'$'s (for all contracted chains) 
glues with the quasimap on
$\cC\setminus D$ to give the required quasistable toric quasimap on $\cC'$.
 \hfill$\Box$

\section{Virtual smoothness}\label{virt}  The goal of this section is
to explain why the obstruction theory provided by the construction of $\cTX$ given
in Theorem \ref{construction} is the natural one, in the sense that it extends the
obstruction theory of maps to $X_\Sigma$.
We use the formalism of perfect obstruction theory and  virtual classes
from \cite{BF}.

\subsection{Maps from a fixed curve} For a fixed nodal curve $C$, consider the moduli scheme $Mor_\beta(C,X_\Sigma)$
of maps to $X_\Sigma$ of class $\beta$. It has a perfect obstruction theory given by the complex
\begin{equation}\label{obstheory} (R^\bullet p_*f^*T_{X_\Sigma})^\vee,\end{equation}
where $$f:C\times Mor_\beta(C,X_\Sigma)\lra X_\Sigma$$ is the universal map and
$$p:C\times Mor_\beta(C,X_\Sigma)\lra Mor_\beta(C,X_\Sigma)$$ is the projection.

\subsection{Euler sequence} The tangent bundle of $X_\Sigma$ fits into an Euler sequence
\begin{equation}\label{Euler sequence 1}
0\lra\cO_{X_\Sigma}^{\oplus r}\lra\oplus_\rho\cO_{X_\Sigma}(D_\rho)\lra T_{X_\Sigma}\lra 0.
\end{equation}
The fastest way to see this is by taking the distinguished triangle of absolute and relative cotangent complexes
for the morphism of Artin stacks
$$X_\Sigma=[(\CC^{\Sigma(1)}\setminus Z(\Sigma))/(\CC^*)^r]\lra [\mathrm{pt}/(\CC^*)^r]=B(\CC^*)^r.$$
The first map in the sequence is the derivative of the action, hence it is induced by (the 
transpose of) the matrix $A$.

Let $$(\cL_\rho,u_\rho),\;\;\; \rho\in\Sigma(1)$$ be the universal line bundles and sections on the universal curve
$$\pi_\cT:\cC_\cT\lra\cTX.$$ 

We have a Cartesian diagram

$$\begin{array}{ccc} \cC_\cT&\lra &\cC_{\mathfrak{P}ic}\\ \\
\;\;\;\;\; \downarrow \pi_\cT& &\pi_{\mathfrak{P}ic}\downarrow \;\;\;\;\; \\ \\
\cTX&\stackrel\mu\lra &\mathfrak{P}ic^r_\beta
\end{array}$$
and the bundles $\cL_\rho$ are the pull-backs of the universal bundles $\cQ_\rho$ on
$\cC_{\mathfrak{P}ic}$.

The matrix $A$ gives a morphism
$$\cO_{\cC_\cT}^{\oplus r}\lra\oplus_\rho\cL_\rho,\;\;\;  1_j\mapsto \sum_\rho a_{j\rho}u_\rho,$$
which is injective since the sections $u_\rho$ satisfy the nondegeneracy condition. Hence, denoting the
quotient by $\cF$ 
we obtain an Euler sequence
on the universal curve $\cC_\cT$ as well:
\begin{equation}\label{Euler sequence}0\lra\cO_{\cC_\cT}^{\oplus r}\lra\oplus_\rho\cL_\rho\lra\cF\lra 0
\end{equation}
It induces a distinguished triangle
$$R^\bullet\pi_*(\cO_{\cC_\cT}^{\oplus r})\lra R^\bullet\pi_*(\oplus_\rho\cL_\rho)\lra R^\bullet\pi_*(\cF),$$
where $\pi=\pi_{\cT}$.

\subsection{Obstruction theory for $\cTX$} Recall from \S\ref{new-construction} that we have a diagram

\begin{equation}\begin{array}{ccc} \label{diagram2}
\cT=\cTX&\stackrel\iota\hookrightarrow &\cE\\ \\
 &\mu\searrow  &\downarrow q\\ \\
 &  &\mathfrak{P}ic^r_\beta\\ \\
  & &\; \downarrow \phi \\ \\
 & &  \fMgk\\
\end{array}\end{equation} 
with $\iota$ the embedding as the zero locus of a section of the bundle 
$$\cV:= q^*\left (\oplus_{\rho\in\Sigma(1)}\pi_*(\cV_\rho)\right )$$ from (\ref{bundle}).
(To unburden the notation, we drop from now on the subscripts for the universal curves
and the projections $\pi$; this should not lead to confusion.)

Hence we obtain a perfect obstruction theory 

$$E^\bullet_\cT:=[\cV^\vee |_{\cT}\lra\Omega^1_\cE |_{\cT}]$$
for $\cTX$. Since the map $q$ is smooth, with relative cotangent bundle
$$\Omega^1_{\cE/\mathfrak{P}ic}=\left (q^*\pi_*(\oplus_\rho \cQ_\rho(m_1))\right )^\vee$$ 
and $\cV_\rho$ is the cokernel of 
$$0\lra\cQ_\rho\lra\cQ_\rho(m_1),$$
it follows that 
the $\mu$-relative obstruction theory of $\cTX$ over $\mathfrak{P}ic^r$ is (quasi-isomorphic to)

$$ E^\bullet_{\cT/\mathfrak{P}ic^r}=\left( R^\bullet \pi_*(\oplus_\rho \cL_\rho)\right)^\vee .$$

Finally, using the Euler sequence (\ref{Euler sequence}) and the fact
that 
$$\mu^*L_{\phi}[1]=\left( R^\bullet\pi_*(\cO^{\oplus r})\right)^\vee ,$$ 
we obtain the $\nu$-relative obstruction theory
$$ E^\bullet_{\cT/\fMgk}=\left( R^\bullet \pi_*(\cF)\right)^\vee .$$

The restriction of the morphism $\nu=\phi\circ\mu$ to the open substack 
$\cTX^\circ$ parametrizing honest maps to $X_\Sigma$ has fibers $Mor_\beta(C,X_\Sigma)$. 
From the two Euler sequences (\ref{Euler sequence 1}) and (\ref{Euler sequence}) we see that the sheaf $\cF$ on $\cC|_{\cTX^\circ}$ is the
pull-back of the tangent bundle of $X_\Sigma$ via the universal map. Hence
the restriction of $ E^\bullet_{\cT/\fMgk}$ to 
$\cTX^\circ$ agrees with the relative perfect obstruction theory over $\fMgk$ given by (\ref{obstheory}).

Finally, we note that the virtual dimension of $\cTX$ is easily calculated to be
$$(1-g)(\dim  X_\Sigma-3)+k+\int_\beta c_1(T_{X_\Sigma}),$$
the same as the virtual dimension of the moduli space of stable maps $\overline{M}_{g,k}(X_\Sigma,\beta)$.

\section{Quasimap integrals}

\subsection{Evaluation maps and descendent integrals} Let $$\pi:\cC\lra\cT_{g,k}(X_\Sigma,\beta)$$ be the universal curve, with
sections $$s_i:\cT_{g,k}(X_\Sigma,\beta)\lra\cC,\;\;\; i=1,\dots,k.$$
The cotangent line bundles at the markings are defined as usual by pulling back the relative dualizing sheaf via the sections
$${\mathbb L}_i:=s_i^*(\omega_{\cC/\cT}) .$$
We put $\hat{\psi}_i:=c_1({\mathbb L}_i)$.

Since the base points cannot occur at markings, $\cT_{g,k}(X_\Sigma,\beta)$ comes
with the well-defined evaluation maps
$$\hat{ev}_i:\cT_{g,k}(X_\Sigma,\beta)\lra X_\Sigma , i=1,\dots ,k.$$

A system of quasimap descendent integrals is defined via integration against the virtual class: for cohomology classes
$\gamma_1,\dots ,\gamma_k\in H^*(X_\Sigma,\QQ)$ and nonnegative integers
$n_1,\dots,n_k$, we put
$$\langle \tau_{n_1}(\gamma_1),\dots,\tau_{n_k}(\gamma_k)\rangle_{g,k,\beta}^{quasi}:=
\int_{[\cT_{g,k}(X_\Sigma,\beta)]^{\mathrm{vir}}}\prod_{i=1}^k\hat{\psi}_i^{n_i}\hat{ev}_i^*(\gamma_i).$$
\bigskip

\subsection{Torus equivariant theory }

Let $\bT=\Hom(\ZZ^{\Sigma(1)},\CC^*)\cong (\CC^*)^l$ denote the ``big" torus acting on $X_\Sigma$.
and let
$$H^*_\bT({\mathrm {pt}},\CC)=H^*(B\bT,\CC)=\CC[\lambda_1,\dots,\lambda_l],$$
be the $\bT$-equivariant cohomology of a point.

The action of $\bT$ on $X_\Sigma$ induces an action on $\cTX$ for which the obstruction theory and the
virtual class are equivariant. Hence, with the same definition, we get a ($\CC[\lambda_1,\dots,\lambda_l]$-valued)
system of $\bT$-equivariant quasimap descendent integrals. The description of the fixed point loci in
$\cTX$ is similar to the one for moduli stable quotients in \cite{MOP}. Together with the
expressions for the obstruction theory derived in \S \ref{virt}, this allows for computing quasimap integrals via the virtual localization formula.
\subsection{Splitting and Cohomological Field Theory} 
The boundary divisors of $\cTX$ exhibit the same recursive structure as in the case of moduli 
of stable maps.
It is straightforward to check that the virtual classes are compatible as well, so that
the analogue of the Splitting Axiom in Gromov-Witten theory holds for quasimap integrals.
Moreover, via the forgetful map
$$\cTX\lra\overline{M}_{g,k}$$
we can define ``quasimap classes" and therefore our virtually smooth moduli spaces
give rise to a Cohomological Field Theory (CohFT for short) on $H^*(X_\Sigma ,\QQ)$. 

\subsection{No contraction maps} Note that the universal curve $\cC_\cT$ over $\cTX$
is {\it not} isomorphic to $\cT_{g,k+1}(X_\Sigma,\beta)$. In fact, in general there are no contraction maps
which forget one of the markings between the quasimap moduli spaces. This means that the proofs of the 
fundamental class and divisor axioms in Gromov- Witten theory cannot be extended and some easy examples show that
these axioms do not hold for quasimap integrals. (We believe, however, that the quasimap and Gromov-Witten integrals coincide 
when $X_\Sigma$ is Fano, see Conjecture \ref{I=J}, so that the axioms themselves hold in this case.)

\begin{EG}

The Hirzebruch surface $\mathbb{F} _n$ is defined by a fan in $N_\RR = \RR ^2$ with $\rho _1 = (1,0),
\rho _2 = (-1,-n), \rho _3 = (0,1), \rho _4 = (0,-1)$. Hence, the matrix $A$ is given
by \[ \left(\begin{array}{cccc} 1 & 1 & n & 0 \\
                                            0 & 0 & 1 & 1 \end{array}\right). \]
We let $D_i=D_{\rho_i}$. Note that $D_4\cdot D_4 = -n$ and that $\mathbb{F}_n$ is not Fano for $n\geq 2$.
Let $n=2$. By the string axiom, in Gromov-Witten theory,
$$\langle 1, D_2 \rangle _{0,2,\beta =D_4} = 0.$$ However, 
one computes using the virtual localization formula that
$$\langle 1, D_2 \rangle _{0,2,\beta = D_4}^{quasi} = -1,\;\;\;
\langle 1, D_2, D_2 \rangle _{0,3, \beta = D_4} ^{quasi} = 0.$$ 
\end{EG}
Hence the
quasimap CohFT is in general different from the Gromov-Witten CohFT.

\section{Quasimaps with one parametrized component and the big $I$-function of $X_\Sigma$}\label{big-I-fcn}

\subsection{Reminder on $J$-functions} For a smooth projective variety $X$
(we assume for simplicity that the odd cohomology of $X$ vanishes), a formalism
due to Givental, see \cite{Giv2}, \cite{CG}, encodes the genus zero Gromov-Witten theory (with descendents) of $X$ into
a Lagrangian cone $\mathfrak{L}_X$ with special properties in a symplectic vector space. In particular,
the theory can be recovered from a special point $J_X$ of this cone, which is a generating function for
Gromov-Witten invariants of $X$ with descendent insertions only at the first marking. Precisely, the
{\it (big) $J$-function of $X$} is defined to be the cohomology valued function

\begin{equation}\label{J-fcn}
J_X({\bf t},z):=1+\frac{\bf t}{z}+\sum_{\beta}Q^{\beta} 
\sum_{k\geq 0}\frac{1}{k!}(ev_1)_*\left (\frac{[\overline{M}_{0,k+1}(X,\beta)]^{\mathrm{vir}}}{z(z-\psi)}\prod_{j=2}^{k+1}ev_j^*({\bf t})\right ).
\end{equation}
Here ${\bf t}\in H^{*}(X,\CC)$, $z$ is a formal variable, $\beta$ runs over 
the integral points in the cone of effective curves on $X$, $Q^\beta$ are Novikov variables, 
$\overline{M}_{0,k+1}(X,\beta)$ is the moduli space of stable maps to $X$, $ev_i$ are the evaluation maps
at the markings, and $\psi=\psi_1$ is the cotangent class at the first marking. The unstable terms with $\beta=0$ and $k\leq 1$ are omitted from the sum.

The small $J$-function is obtained by restricting ${\bf t}$ to $H^{\leq 2}(X,\CC)$. It is a generating function for $1$-point genus zero invariants of $X$.

There is a well-known description of the $J$-function in terms of equivariant residues on the so-called {\it graph spaces}
$$G_k(X,\beta):=\overline{M}_{0,k}(\PP^1\times X, (1,\beta)).
$$
Namely, consider the $\CC^*$-action on the graph space induced by the action on the factor $\PP^1$ in $\PP^1\times X$.
Given a stable map 
$$(C,p_2,\dots ,p_{k+1}, f:C\lra \PP^1\times X)$$
(the unusual numbering of markings is chosen for later matching with the formula \ref{J-fcn}),
$C$ has a distinguished component $C_0$ such that $q_{\PP^1}\circ f |_{C_0}$ is an isomorphism --
here $q_{\PP^1}$ and $q_X$ are the projections from $\PP^1\times X$ to the factors.
We identify $C_0$ with $\PP^1$ via the isomorphism above. In particular, we have the points $0,\infty\in C_0$.

A stable map is fixed by the $\CC^*$-action if an only if it satisfies the following properties:

\vskip.1in

$(G1)$ $C_0\cap \overline{(C\setminus C_0)}\subset\{ 0,\infty\}$.

$(G2)$ There are no markings on $C_0\setminus \{0,\infty\}$.

$(G3)$  The map $q_X\circ f |_{C_0}$ is constant. In other words, the class $\beta$ is concentrated
above $0$ and $\infty$.
\vskip.1in

The ``components" of the fixed point locus (the quotation marks indicate that they need not be connected components) are therefore all
isomorphic to products
$$\overline{M}_{0,k_1+1}(X,\beta_1)\times_X \overline{M}_{0,k_2+1}(X,\beta_2).$$
We'll be interested in only one such component, denoted $F_0$, for which $C_0\cap \overline{(C\setminus C_0)}=\{ 0\}$
and $C\setminus C_0$ contains all markings. We have 
$$F_0\cong \overline{M}_{0,k+1}(X,\beta),$$
with the additional marking $p_1$ parametrizing the point on $\overline{(C\setminus C_0)}$ where $C_0$ is attached.

When restricted to $F_0$, the $\CC^*$-equivariant obstruction theory of the graph space has fixed part equal to
the usual obstruction theory of $\overline{M}_{0,k+1}(X,\beta)$, while the moving part (i.e., the virtual normal bundle of $F_0$)
has Euler class

$$e(N_{F_0/G}^{\mathrm{vir}})=z(z-\psi).$$
Here $z$ is the generator of the $\CC^*$-equivariant cohomology of a point,
$$H^*_{\CC^*}(\mathrm{pt},\CC)=\CC[z],$$
and $\psi$ is the cotangent line class at the first marking.
It follows that the class
$$\frac{[\overline{M}_{0,k+1}(X,\beta)]^{\mathrm{vir}}}{z(z-\psi)}\prod_{j=2}^{k+1}ev_j^*({\bf t})$$
from the formula (\ref{J-fcn}) is precisely the residue at $F_0$ of the class
$$[G_k(X,\beta)]^{\mathrm{vir}}\prod_{j=2}^{k+1}ev_j^*({\bf t})$$
in the virtual localization formula of \cite{GP}.

A theorem of Givental, \cite{G}, states that when $X=X_\Sigma$ is a smooth projective toric variety with $c_1(T_{X_\Sigma})\geq 0$\footnote{This restriction
has been subsequently removed in \cite{Iritani}. The result has also been 
recently reproved and generalized to toric fibrations in \cite{Brown}.}, its small $J$-function
can be calculated in terms of another cohomology-valued function $I_{X_\Sigma}$ which is also given by
a certain $\CC^*$-equivariant residue on a moduli space of toric quasimaps from $\PP^1$ to $X_\Sigma$. We will call this function
the {\it small} I-function of $X_\Sigma$ (its precise definition will be recalled in Definition \ref{small-I} below). 

Our goal here is to show how moduli spaces of stable toric quasimaps can be used to give a geometric
definition of the {\it big} $I$-function of a toric variety and to present some results and conjectures relating it
with the big $J$-function.

\subsection{Stable toric quasimaps with one parametrized component} 
Assume we are given a toric variety $X_\Sigma$ with a polarization as in \S 2, 
integers $g,k\geq 0$, an effective class $\beta\in H_2(X_\Sigma,\ZZ)$, and a fixed
nonsingular projective curve $D$. Let ${\bf T}$ be the torus acting on $X_\Sigma$. The theory developed in this
subsection should be understood as the $\bT$-equivariant theory. In particular, Conjecture \ref{I=J} and 
Theorem \ref{proj space} below express relationships between the $\bT$-equivariant quasimap integrals and 
the $\bT$-equivariant Gromov-Witten invariants of $X_\Sigma$.

However, since the torus $\bT$ plays no role in
the arguments, it will be dropped from the notation.

\begin{Def} \label{parametrized qmap}
A stable, $k$-pointed, toric quasimap of genus $g$ and class $\beta$ to $(X_\Sigma,D)$ is specified by the data
$$( (C,p_1,\dots ,p_k), \{ {\mathcal P}_i\; |\; i=1,\dots r\}, \{ u_\rho\}_{\rho\in\Sigma(1)}, \varphi),$$ 
where

\begin{itemize}

\item $(C,p_1,\dots ,p_k)$ is a connected, at most nodal, projective curve of genus $g$, and $p_i$ are distinct nonsingular points of $C$,

\item ${\mathcal P}_i$ are line bundles on $C$, of degrees $f_i:=\int_\beta c_1({\mathcal L}_i)$,

\item $u_\rho \in \Gamma (C, L_\rho )$ are global sections of the line bundles $$L_\rho:=\bigotimes_{i=1}^r{\mathcal P}_i^{\otimes a_{i\rho}}, $$

\item $\varphi : C\lra D$ is a regular map,

 \end{itemize}
 subject to the conditions:
 \begin{enumerate}
 
 \item (parametrized component) $\varphi_*[C]=[D]$. Equivalently, there is a distinguished component $C_0$ of $C$ such that
 $\varphi$ restricts to an isomorphism $C_0\cong D$ and $\varphi (C\setminus C_0)$ is zero-dimensional (or empty, if $C=C_0$). 
  
 \item (nondegeneracy) There is a finite (possibly empty) set of nonsingular points $B\subset C$,
disjoint from the markings on $C$,
such that for every $y\in C\setminus B$,  there exists a maximal cone $\sigma \in \Sigma _{\mathrm{max}}$ 
with $u_\rho (y)\ne 0$, $\forall \rho \not\subset \sigma$,

\item (stability) $\omega _{\tilde{C}}(p_1+\dots +p_k) \ot {\mathcal L}^\epsilon $ is ample for every rational $\epsilon > 0$, where
$\tilde{C}$ is the closure of $C\setminus C_0$. 

  \end{enumerate}

  Here the line bundle $\cL$
  in the stability condition is
  $$\cL=\otimes_{\rho\in\Sigma(1)}(\otimes_{j=1}^r \cP_j^{\otimes a_{j\rho}})^{\otimes \alpha_\rho} .$$

\end{Def}

We denote by

$$\cT_{g,k}(X_\Sigma,\beta; D)$$
the stack parametrizing the stable toric quasimaps in Definition \ref{parametrized qmap}.
Note that it is empty if $g< g(D)$. However, since there is no stability imposed on the distinguished component, the 
inequality $2g-2+k\geq 0$ is not required anymore.
By straightforward extensions of the arguments in \S 2-\S 5 we obtain the following.

\begin{Thm} 
$\cT_{g,k}(X_\Sigma,\beta; D)$ is a proper Deligne-Mumford stack of finite type, with a perfect obstruction theory.
\end{Thm}

Even though we have given the definition and the existence result in the general context, our interest is in the
case $g=0$ and (necessarily) $D=\PP^1$, to which we restrict from now on. These are the quasimap analogues of
the graph spaces in the usual Gromov-Witten theory.

The underlying curve for
a point in $\cT_{0,k}(X_\Sigma,\beta; \PP^1)$ is a $k$-pointed 
tree of rational curves with one parametrized component $C_0\cong\PP^1$, that is, a
point in the stack $\widetilde{\PP^1[k]}$, the Fulton-MacPherson
space of  (not necessarily stable) configurations of $k$ distinct points on $\PP^1$. This is a smooth Artin
stack of infinite type (see \S 2.8 in \cite{KKO}).
 
Forgetting the rest of the data we obtain a morphism of stacks
$$\nu:\cTXP\lra \widetilde{\PP^1[k]}.$$ As in \S \ref{virt}, one sees easily that we have an Euler sequence
$$0\lra\cO_{\cC}^{\oplus r}\lra\oplus_\rho\cL_\rho\lra\cF\lra 0$$
on the universal curve
$$\pi: \cC\lra \cTXP$$ and that the $\nu$-relative perfect obstruction theory is 
given by
$$\left( R^\bullet \pi_*\cF\right )^\vee .$$
It follows that the virtual dimension is
$$\mathrm{vdim}\; \cTXP=\dim\; X_{\Sigma}+k+\int_\beta c_1(T_{X_\Sigma}).$$

\begin{EG} The moduli space $\cT_{0,0}(X_\Sigma,\beta ;\PP^1)$ is the toric compactification of Givental
and Morrison-Plesser of $Mor_\beta(\PP^1,X_\Sigma)$.
\end{EG}

\begin{Rmk} It is obvious that the graph space (for an arbitrary smooth projective $X$) can be viewed as the moduli space 
$$\overline{M}_{0,k}(X,\beta;\PP^1)$$
of stable
maps to $X$ with one parametrized component, whose points are described as
$$(C,p_1,\dots,p_k,f:C\ra X,\varphi :C\ra\PP^1),\;\; f_*[C]=\beta,\;\;\varphi_*[C]=[\PP^1].$$
Forgetting the map $f$ gives a morphism to
$\widetilde{\PP^1[k]}$, for which the relative obstruction theory is given by
the complex $\left( R^\bullet \pi_*(f^*T_X)\right)^\vee$.
We also have the smooth morphism of stacks
$$\widetilde{\PP^1[k]}\lra\mathfrak{M}_{0,k}$$
which forgets the map $\phi$. Its relative cotangent bundle is given by $\left( R^0\pi_*\varphi^*T_{\PP^1}\right )^\vee$.
It follows that the relative obstruction theory for the
composed forgetful map to $\mathfrak{M}_{0,k}$ is the obstruction theory for the graph space.
\end{Rmk}

\begin{Prop} Assume that $X_\Sigma$ has the property that the divisors $D_\rho$ are all nef. 
Then $R^1\pi_*\cF=0$.
\end{Prop}

\begin{proof} The assumption implies that given a
stable quasimap 
$$( (C,p_1,\dots ,p_k), \{ {\mathcal P}_i\; |\; i=1,\dots r\}, \{ u_\rho\}_{\rho\in\Sigma(1)}, \varphi),$$ 
we have
$\deg (L_\rho |_{C'})\geq 0$ for every irreducible component $C'$ of $C$. Hence 
$$H^1(C,\oplus_\rho L_\rho)=0$$
and the result follows from the Euler sequence.
\end{proof}

\begin{EG} If $X_\Sigma$ is a product of projective spaces then the moduli spaces
$\cT_{0,k}(X_\Sigma, \beta)$ and $\cTXP$ are smooth, irreducible Deligne-Mumford stacks
and the virtual classes are the usual fundamental classes.

Also, the argument in \S 3.3 of \cite{MOP} shows that the
same is true for the elliptic moduli space $\cT_{1,0}(X_\Sigma, \beta)$
(still under the assumption that $X_\Sigma$ is a product of projective spaces).
\end{EG}

The $\CC^*$-action on $\PP^1$ induces an action on $\cTXP$, whose fixed loci have
descriptions similar to the ones in the case of graph spaces. Namely, a stable quasimap

$$( (C,p_2,\dots ,p_{k+1}), \{ {\mathcal P}_i\; |\; i=1,\dots r\}, \{ u_\rho\}_{\rho\in\Sigma(1)}, \varphi)$$
is fixed by $\CC^*$ if and only if it satisfies the properties
\vskip.1in

$(Q1)$ $C_0\cap \overline{(C\setminus C_0)}\subset\{ 0,\infty\}$; here $C_0$ is identified with $\PP^1$ via $\varphi$.

$(Q2)$ There are no markings on $C_0\setminus \{0,\infty\}$.

$(Q3)$ The curve class $\beta$ is ``concentrated at $0$ or $\infty$". Precisely, this means the following: there are no
base points on $C_0\setminus \{0,\infty\}$ and the resulting map
$$C_0\setminus \{0,\infty\} \lra X_\Sigma$$
is constant.
\vskip.1in

Again, we are only interested in the component $\hat{F_0}$ of the fixed point locus for which the curve class $\beta$
is concentrated only at $0\in C_0$, i.e., the component parametrizing $\CC^*$-fixed quasimaps for which
$C_0\cap \overline{(C\setminus C_0)}=\{0\}$ and there is no base point or marking at $\infty$. The cases $k=0$ and $k\geq 1$ exhibit different behavior.

{\it Case $k\geq 1$}:  We have $\hat{F_0}\cong \cT_{0,k+1}(X_\Sigma,\beta)$ and the (virtual) codimension in $\cTXP$ is equal to $2$.

{\it Case $k=0$}: In this case the only parameter for the $\CC^*$-fixed stable quasimap is the image of $C_0\setminus \{0\} \lra X_\Sigma$. 

We have $\hat{F_0}\cong X_\Sigma$ if there are no invariant divisors $D_\rho$ which intersect negatively the curve class $\beta$
and 
$$\hat{F_0}\cong\bigcap_{\{\rho\;|\; \int_\beta c_1(\cO(D_\rho))<0\}}D_\rho$$
in general. The virtual codimension is equal to
$$\int_\beta c_1(T_{X_\Sigma})+\#\{ \rho\;|\; \int_\beta c_1(\cO_{X_\Sigma}(D_\rho))<0\}.$$
(Note that the moduli space $\cT_{0,1}(X_\Sigma,\beta)$ doesn't exist!)

In the first case, we have the following.

\begin{Lemma}\label{residue} The restriction of the (absolute) obstruction theory of $\cTXP$ to $\hat{F_0}$ 
has fixed part equal to the 
(absolute)
obstruction theory of $\cT_{0,k+1}(X_\Sigma,\beta)$. Furthermore, the Euler class of the virtual normal bundle is
$$e(N_{\hat{F_0}/\cT}^{\mathrm{vir}})=z(z-\hat{\psi}),$$
where $\hat{\psi}=\hat{\psi_1}$.

\end{Lemma}

\begin{proof} Follows easily from the cartesian diagram (in the category of $\CC^*$-spaces)
$$\begin{array}{ccc}
\cT_{0,k+1}(X_\Sigma,\beta)&\hookrightarrow &\cTXP\\ \\
\downarrow & &\downarrow \\ \\
\mathfrak{M}_{0,k+1} &\hookrightarrow &\widetilde{\PP^1[k]}
\end{array} ,$$
the fact that the relative obstruction theories for the two vertical maps are both equal to $\left( R^\bullet\pi_*(\cF)\right)^\vee$,
and a straightforward computation for the equivariant Euler class of the normal bundle of $\mathfrak{M}_{0,k+1}$ in
$\widetilde{\PP^1[k]}$. Details are left to the reader.
\end{proof}

In the second case Givental has shown that
the restriction of the (absolute) obstruction theory of $\cT_{0,0}(X_\Sigma,\beta;\PP^1)$ to $\hat{F_0}$ has fixed part
equal to $T_{\hat{F_0}}$, while the moving part has Euler class
$$ e(N_{\hat{F_0}/\cT}^{\mathrm{vir}})=e(N_{\hat{F_0}/X_{\Sigma}})\prod_{\rho\in\Sigma(1)}\frac{\prod_{j=-\infty}^{\int_\beta D_\rho}(D_\rho+jz) }
{\prod_{j=-\infty}^0(D_\rho+jz)}.$$
Hence the push-forward of the residue of $[\cT_{0,0}(X_\Sigma,\beta;\PP^1)]^{\mathrm{vir}}$ at $\hat{F_0}$ via $i_0:\hat{F_0}\hookrightarrow X_\Sigma$
is
$$(i_0)_*\frac{1}{e(N_{\hat{F_0}/\cT}^{\mathrm{vir}})}=
\prod_{\rho\in\Sigma(1)}\frac{\prod_{j=-\infty}^0(D_\rho+jz)}{\prod_{j=-\infty}^{\int_\beta D_\rho}(D_\rho+jz)}.$$

Givental defined the small $I$-function of $X_\Sigma$ as the sum over the curve classes $\beta$
of these residues, up to an exponential factor. Precisely,

\begin{Def} \label{small-I} The small $I$-function is
$$ I_{X_\Sigma}^{\mathrm{small}}({\bf t},z):= e^{\frac{{\bf t}}{z}}\sum_{\beta}Q^\beta e^{\int_\beta {\bf t}}
\prod_{\rho\in\Sigma(1)}\frac{\prod_{j=-\infty}^0(D_\rho+jz)}{\prod_{j=-\infty}^{\int_\beta D_\rho}(D_\rho+jz)}$$
where ${\bf t}\in H^0(X_\Sigma,\CC)\oplus H^2(X_\Sigma,\CC).$

\end{Def}

The exponential factor is introduced in view
of Givental's result comparing $I^{\mathrm{small}}_{X_\Sigma}$ with the small $J$-function. This is because the string and divisor equations in Gromov-Witten
theory imply that the $k\geq 1,\; \beta\neq 0$ part of the big $J$-function (\ref{J-fcn}) collapses to 
$$\sum_{\beta\neq 0} Q^\beta (e^{{\bf t}/z}e^{\int_\beta{\bf t}}-1)(ev_1)_*\left( \frac{[\overline{M}_{0,1}(X,\beta)]^{\mathrm{vir}}}{z(z-\psi)}\right)
$$
when ${\bf t}$ is restricted to $H^{\leq 2}$, so that the small $J$-function is
$$J_{X_\Sigma}^{\mathrm{small}}=
e^{\frac{{\bf t}}{z}}\sum_{\beta}Q^\beta e^{\int_\beta {\bf t}}(ev_1)_*\left( \frac{[\overline{M}_{0,1}(X,\beta)]^{\mathrm{vir}}}{z(z-\psi)}\right).
$$

Hence the difference between the small $J$ and $I$ functions is simply that the push-forwards of residues at $F_0$ 
of the virtual class of the
spaces of stable maps with one parametrized component $\overline{M}_{0,0}(X_{\Sigma},\beta;\PP^1)$
are replaced with the residues at $\hat{F_0}$ of the virtual class of the analogous quasimap spaces
$\cT_{0,0}(X_\Sigma,\beta;\PP^1)$.

The following Definition and Conjecture were made jointly with Rahul Pandharipande.

\begin{Def} The big $I$-function of the toric variety $X_{\Sigma}$ is
\begin{align}
I_{X_\Sigma}=&1+\frac{{\bf t}}{z}+\sum_{\beta\neq 0}Q^\beta \prod_{\rho\in\Sigma(1)}\frac{\prod_{j=-\infty}^0(D_\rho+jz)}
{\prod_{j=-\infty}^{\int_\beta D_\rho}(D_\rho+jz)}\nonumber \\
&+\sum_{\beta}Q^{\beta} 
\sum_{k\geq 1}\frac{1}{k!}(\hat{ev}_1)_*\left (\frac{[\cT_{0,k+1}(X_\Sigma,\beta)]^{\mathrm{vir}}}{z(z-\hat{\psi})}\prod_{j=2}^{k+1}\hat{ev}_j^*({\bf t})\right ).\nonumber
\end{align}

\end{Def}

Note that the $k\geq 1$ terms are formally the same as the ones in (\ref{J-fcn}), while the $k=0$ part comes from the 
correspondence between small $I$ and $J$. Lemma \ref{residue} says that $I$ is equal to a sum of $\CC^*$-residues in
exactly the same way as $J$.

\begin{Conj}\label{I=J} The function $I_{X_\Sigma}$ is in the Lagrangian cone $\mathfrak{L}_{X_\Sigma}$
defined by the Gromov-Witten theory of $X_\Sigma$, so that there exists a change of variable
(``mirror transformation") $\tau({\bf t})$ with 
$$J_{X_\Sigma}({\bf t}, z)=I_{X_\Sigma}(\tau,z)+{\mathrm{linear}\; \mathrm{combination}\;
\mathrm{of}\; \mathrm{derivatives}\; \mathrm{of}\; }I_{X_\Sigma}(\tau,z). 
$$
Furthermore, if
$X_\Sigma$ is Fano, then $I_{X_\Sigma}({\bf t}, z)=J_{X_\Sigma}({\bf t},z)$.
\end{Conj}

We plan to address this conjecture elsewhere. Here we only note the following result as supporting evidence.

\begin{Thm}\label{proj space} $I_{\PP^n}=J_{\PP^n}$.
\end{Thm}

\begin{proof} The argument is an extension of the one given in \cite{Bertram} for small $J$ and $I$-functions of
projective spaces. We write $d$ for the curve class $\beta=d[\mathrm{line}]$.
The moduli spaces $\overline{M}_{0,k}(\PP^n,d;\PP^1)$ and $\cT_{0,k}(\PP^n,d;\PP^1)$ are
smooth, of the same (expected) dimension. Furthermore, there is a $\CC^*$-equivariant birational {\it regular} map
$$\Phi : \overline{M}_{0,k}(\PP^n,d;\PP^1)\lra \cT_{0,k}(\PP^n,d;\PP^1)$$
such that 
\begin{enumerate}
\item the only fixed component mapped by $\Phi$ to $\hat{F_0}$ is $F_0$.

\item when $k=0$, the restriction $\Phi |_{F_0}$ coincides with the evaluation map $ev_1:\overline{M}_{0,1}(\PP^n,d)\ra\PP^n$.

\item when $k\geq 1$, $\Phi |_{F_0}$ is birational onto $\hat{F_0}$ and $ev_j =\hat{ev_j}\circ \Phi |_{F_0}$ for $j=1,\dots ,k+1$.

\end{enumerate}
(For $k=0$ the map $\Phi$ was constructed in \cite{Giv3}; for the general case, see \cite{MOP}.)

Our statement is now a consequence of the localization theorem (in the form of
``correspondence of residues"), in view of birationality of $\Phi$ and
the expressions of $I$ and $J$ in terms of residues.
\end{proof}
\begin{Rmk} 
$(a)$ The above theorem is a special case of the main result of \cite{MOP}. However, our proof is different.

$(b)$ The theorem and its proof can be immediately extended to the case of products of projective spaces.
However, for a general $X_\Sigma$ the contraction map $\Phi$ does not exist, and even if it did, it is not
obvious that the push-forward would preserve the virtual classes. Hence the argument does not
extend immediately
to prove the Fano case of conjecture \ref{I=J}.
\end{Rmk}

\end{document}